%% file: Accepted_author_manuscript_Pelessoni_Vicig.tex
\documentclass{amsart}

\usepackage{amssymb,amstext,amsmath,amsthm}
\usepackage{graphicx}

\input{macros}

\title{The Goodman-Nguyen Relation within Imprecise Probability Theory}

\newtheorem{lemma}{Lemma}
\newtheorem{proposition}{Proposition}

\theoremstyle{remark}
\newtheorem{remark}{Remark}
\theoremstyle{example}
\newtheorem{example}{Example}
\theoremstyle{definition}
\newtheorem{definition}{Definition}

\begin{document}




\title{The Goodman-Nguyen Relation within Imprecise Probability Theory}


\author{RENATO PELESSONI}

\address{DEAMS ``B. de Finetti''\\
University of Trieste\\
Piazzale Europa~1\\
I-34127 Trieste\\
Italy}
\email{renato.pelessoni@econ.units.it}

\author{PAOLO VICIG}

\address{DEAMS ``B. de Finetti''\\
University of Trieste\\
Piazzale Europa~1\\
I-34127 Trieste\\
Italy}
\email{paolo.vicig@econ.units.it}

\begin{abstract}
The Goodman-Nguyen relation is a partial order generalising the implication 
(inclusion) relation to conditional events.
As such,
with precise probabilities it both induces an agreeing probability ordering and is a key tool
in a certain common extension problem.
Most previous work involving this relation is concerned with either conditional event algebras or precise probabilities.
We investigate here its role within imprecise probability theory,
first in the framework of conditional events and then proposing a generalisation of the Goodman-Nguyen relation to
conditional gambles.
It turns out that this relation induces an agreeing ordering on coherent or C-convex
conditional imprecise previsions.
In a standard inferential problem with conditional events,
it lets us determine the natural extension,
as well as an upper extension.
With conditional gambles,
it is useful in deriving a number of inferential inequalities. 

\end{abstract}

\maketitle


\section*{Acknowledgement}
*NOTICE: This is the authors' version of a work that was accepted for publication in the International Journal of Approximate Reasoning. Changes resulting from the publishing process, such as peer review, editing, corrections, structural formatting, and other quality control mechanisms may not be reflected in this document. Changes may have been made to this work since it was submitted for publication. A definitive version was subsequently published in the International Journal of Approximate Reasoning, vol. 55, issue 8, November 2014, pages 1694-–1707, doi:10.1016/j.ijar.2014.06.002

$\copyright$ Copyright Elsevier

http://www.sciencedirect.com/science/article/pii/S0888613X14001017

\section{Introduction}
\label{sec:introduction}
It is well known in probability theory
that relations among events determine elementary probability rules as well as inferential bounds.
Take for instance two non-trivial (i.e. non-impossible, non-certain) events $E$, $F$.
If we know nothing else about $E$, $F$, any probability assignment $P$ on $\{E,F\}$
such that $P(E)\in [0,1]$, $P(F)\in[0,1]$ is consistent (coherent).
Knowing further that $E$ and $F$ are disjoint,
i.e. that $E\wedge F=\varnothing$,
introduces the constraint $P(E)+P(F)\leq 1$.
We can also \emph{infer} that $P(E\vee F)=P(E)+P(F)$.
This inference would not be necessarily justified replacing $P$ with a more general uncertainty measure:
with a coherent lower probability $\lpr$,
for instance,
we could only state that $\lpr(E\vee F)\geq\lpr(E)+\lpr(F)$.

The  \emph{implication} (inclusion) relation $E\Rightarrow F$ ($E\subset F$) is the most prominent example of a relation
whose effects should (in principle) be independent of the uncertainty measure $\mu$ we use.
This is because 
the \emph{monotonicity} requirement
\begin{equation}
\label{eq:EimpliesF}
(E\Rightarrow F) \rightarrow \mu(E)\leq \mu(F)
\end{equation}
is
a very minimal one.
In fact, $E\Rightarrow F$ means that $F$ is
certainly true whenever $E$ is true, but might possibly be true even
in cases when $E$ is false: then obviously $F$ must be at least as
likely as $E$. In fact, \eqref{eq:EimpliesF} holds also when $\mu$
is a coherent lower/upper probability, or a capacity. In the latter
case, it is generally taken as one of the defining properties of
capacities.

The implication relation `$\Rightarrow$' also plays a role in the
following extension problem, a special case of de Finetti's Fundamental
Theorem \cite{deF74}: given a coherent probability $P$ on the set
$\genfield$ of all events (logically) dependent on a given partition
$\prt$, which are its coherent extensions to an additional event
$E\notin \field(\prt)$? The well known answer is that $P(E)$ must be
chosen in a closed interval, $P(E_{*})\leq P(E)\leq P(E^{*})$.
Here the events $E_{*}$, $E^{*}$ are defined using the implication relation (see Definition \ref{def:inner_outer})
and belong to $\genfield$.
Hence, $P(E)$ is bounded by the previous assessment on $\genfield$.

A generalisation of the implication relation to \emph{conditional}
events, the Goodman-Nguyen (in short: GN) relation $\lgn$ was
apparently first introduced in \cite{Goo88}, and some of its
implications for precise conditional probabilities were studied in
\cite{Col93, Col96, Mil97}.

The main purpose of this paper is to further explore the relevance
of the GN relation in more general cases,
extending the previous work in \cite{Pel13}.
Section \ref{sec:preliminaries} recalls some preliminary material,
including a survey of known facts about the GN relation in Section \ref{subsec:GN_relation}.
In Section \ref{sec:GN_relation_ip},
the role of the GN relation with either coherent according to Williams' definition (W-coherent)
or C-convex imprecise probabilities is studied.
Proposition \ref{pro:preserve_inequality} ensures that the generalisation of equation \eqref{eq:EimpliesF},
i.e.
\begin{equation}
\label{eq:defleGN}
A|B\lgn C|D\rightarrow\mu(A|B)\leq\mu(C|D)
\end{equation}
holds for such probabilities.
Section \ref{subsec:natural_extension} considers extensions of W-coherent or C-convex `full' probabilities;
such `full' probabilities are defined on all events
$A|B$ such that both $A$ and $B$ ($B\neq\varnothing$) are logically dependent on some given partition.
Proposition \ref{pro:ext_prob} characterises consistent extensions on an additional event $C|D$;
Propositions \ref{pro:ext_prob_many} and \ref{pro:GN-extensions} characterise
the natural and convex natural extensions on an \emph{arbitrary} set of additional events.
Proposition \ref{pro:upper-extension} regards the upper extensions,
Proposition \ref{pro:dF-coherent_extensions} the special case of precise probabilities.
The investigation of the GN relation with conditional gambles starts in Section \ref{subsec:GN relation conditional events}
with a discussion of its betting interpretation in the standard case of events.
Some aspects not emphasised in the previous literature are highlighted.
This justifies,
together with Proposition \ref{pro:equivalence_GN_rn_events},
our extension (Definition \ref{def:GN_random_numbers}) of the GN relation in Section \ref{subsec:GN_conditional_gambles}.
The extension induces an agreeing ordering on several types of conditional imprecise previsions,
as shown in Section \ref{subsec:Ordering from GN relation}.
A number of special inequalities are also derived in Section \ref{sec:inequalities_GN}.
Section \ref{sec:conclusions} concludes the paper.
\section{Preliminaries}
\label{sec:preliminaries}
In this section we first fix some notation and definitions to be used throughout the paper (Section \ref{subsec:notation}).
We then briefly recall some basic facts from the theory of imprecise probabilities
that will be needed in the sequel (Section \ref{subsec:imprecise_previsions}).
For in-depth studies of these issues,
cf. \cite{deF74, Pel05, Pel09, Wal91, Wei01, Wil07}.
The Goodman-Nguyen relation and known related issues are surveyed in Section \ref{subsec:GN_relation}.

\subsection{Notation and Definitions}
\label{subsec:notation}
In the sequel, following \cite{deF74,Goo88} and others,
we employ the logical rather than the set
theoretical notation for operations with events.
Let $\prt$ be a \emph{partition},
i.e. a set of pairwise disjoint
events whose logical sum (union) is the sure event $\Omega$.
An event $E$ is \emph{logically dependent} on $\prt$ iff $E$ is a
logical sum of events of $\prt$,
$E=\bigvee\limits_{\omega\in\prt:\ \omega\Rightarrow E}\omega$.
The set $\genfield$ of all events logically
dependent on $\prt$ is a field (also called the power set of $\prt$).

When dealing with unconditional events,
probabilities and other uncertainty measures are often assessed on $\genfield$,
where $\prt$ is given.
In general, we may think of eliciting an uncertainty measure on an \emph{arbitrary} set of events
$\mathcal{D}=\{E_i:i\in I\}$.
The set $\mathcal{D}$ is generally not a partition nor the power set of some partition,
but \emph{generates} a partition \emph{$\prt_g$}.
The elements of this partition are given by all the logical products
$\bigwedge_{i\in I}E_i^\prime$, where for each $i\in I$ the symbol $E_i^\prime$
can be replaced by either event $E_i$ or its negation $\nega{E_i}$. 
Note that some products $\bigwedge_{i\in I}E_i^\prime$ may be impossible.
The events in $\mathcal{D}$ belong to $\field(\prt_g)$.

A random number $X$ is also described by a (not uniquely identified) partition.
Typically, we consider for this the canonical partition $\prt_X$.
Its events are $(X=x)$,
i.e. $\prt_X$ is the partition of all possible distinct values $x$ of $X$.
Yet, we might be bound to refer to other,
more refined,
partitions.
For instance,
when describing two random variables $X$, $Y$ at the same time,
we may refer to the partition $\prt_{X,Y}$ whose events are $(X=x\ \wedge\ Y=y)$,
for all jointly possible values of $X$, $Y$.
This is an example of \emph{product partition} 
(other instances will appear in some proofs in Section \ref{sec:GN_relation_ip}).
In general,
given two partitions $\prt$, $\prt^\prime$ with generic elements, respectively,
$\omega$, $\omega^\prime$,
their product partition is $\prt\wedge\prt^\prime=\{\omega\wedge\omega^\prime:\omega\in\prt, \omega^\prime\in\prt^\prime\}$.
Hence $\prt_{X,Y}=\prt_X\wedge\prt_Y$.
The product partition $\prt\wedge\prt^\prime$ is more refined than
both $\prt$ and $\prt^\prime$
(hence, if $E\in\genfield$,
then also $E\in\field(\prt\wedge\prt^\prime$)).
Note that some $\omega\wedge\omega^\prime$ are, in general,
impossible:
the special case $\omega\wedge\omega^\prime\neq\varnothing$,
$\forall\omega\in\prt$, $\forall\omega^\prime\in\prt^\prime$
characterizes the \emph{logical independence}
of $\prt$ and $\prt^\prime$.
We shall not necessarily assume logical independence.

The random numbers we shall deal with in the sequel are all bounded, i.e. they are \emph{gambles}.

In our framework,
conditional events and gambles will be needed.
In terms of a truth
table, a conditional event $A|B$ can be thought of as true, when $A$
and $B$ are true, false when $A$ is false and $B$ true, undefined
when $B$ is false. It ensues that $A|B$ and $A\wedge B|B$ have the
same logical values, i.e. $A|B={A\wedge B}|B$.

Given a partition $\prt$ (describing $X$),
a conditional gamble $X|B$,
$B\in\genfield-\{\varnothing\}$, takes up the values $X(\omega)$,
for $\omega\in\prt$, $\omega\Rightarrow B$, is undefined for
$\omega\Rightarrow\nega{B}$. When $B=\Omega$, $X|\Omega=X$ is an unconditional gamble.
The \emph{indicator} $I_A$ of an event
$A$ is the simplest non-trivial gamble.
We shall often denote
$A$ and its indicator $I_A$ with the same letter $A$. Note that
$A\Rightarrow B$ is equivalent to $I_A\leq I_B$:
the implication relation between events corresponds to the weak inequality between their indicators.

\subsection{Imprecise Previsions}
\label{subsec:imprecise_previsions}
A \emph{lower prevision} $\lpr$ on a set $\setran$ of conditional gambles is a map $\lpr:
\setran\longmapsto\reals$. If $\setran$ has the property $X|B\in \setran\rightarrow
-X|B\in \setran$, the conjugate \emph{upper prevision} of $\lpr$ is defined as
$\upr(X|B)=-\lpr(-X|B)$.
Conjugacy allows employing lower or alternatively upper previsions only.

Several consistency concepts for lower/upper previsions have been introduced in the literature.
An important one is the following:
\begin{definition}
\label{def:W-coherence} A  lower prevision $\lpr : \setran\longmapsto
\reals$ is \emph{W-coherent} iff, for all $n\in\nats$, $\forall
X_0|B_0,\ldots,X_n|B_n\in \setran$, $\forall\ s_0,s_1,\ldots,s_n$ real and
\emph{non-negative}, defining $B=\bigvee_{i=0}^{n} B_{i}$ and
$\lgain=\sum_{i=1}^{n}s_{i}B_{i}(X_{i}-\lpr(X_{i}|B_{i}))-s_{0}B_{0}(X_{0}-\lpr(X_{0}|B_{0}))$,
the following condition holds: $\sup(\lgain|B)\geq 0$.
\end{definition}
This is essentially Williams' definition of coherence \cite{Wil07},
as restated in \cite{Pel09}.
It is equivalent to Walley's definition 7.1.4(b) in \cite{Wal91}
if $\setran$ is made up of a finite number of conditional gambles,
each with finitely many values.
If $X|B=X|\Omega=X$, $\forall X|B\in\setran$,
it reduces to Walley's (unconditional) coherence (\cite{Wal91}, Sec. 2.5.4 (a)).

Like the other consistency concepts we recall in this section,
Definition \ref{def:W-coherence} is axiomatic,
but is customarily given an interpretation in terms of betting schemes.
To outline it,
recall that the conditional gamble $\lgain|B$ is the \emph{gain} from betting
in favour of $X_1|B_1,\ldots, X_n|B_n$ and against $X_0|B_0$
at \emph{stakes} $s_1,\ldots,s_n$ and $s_0$ respectively.
The bet regarding $X_i|B_i$ ($i=0,\ldots,n$) is called off iff $B_i$ is false.
Conditioning $\lgain$ on $B$ requires that at least one bet is effective.

A weaker concept than W-coherence is that of \emph{convex}
conditional lower prevision.
It may be obtained from Definition
\ref{def:W-coherence} by introducing the extra \emph{convexity
constraint} $\sum_{i=1}^{n}s_{i}=s_{0}\ (>0)$. 
\begin{definition}
\label{def:C-convexity}
A  lower prevision $\lpr : \setran\longmapsto
\reals$ is \emph{convex} iff, for all $n\in\nats$, $\forall
X_0|B_0,\ldots,X_n|B_n\in \setran$, $\forall\ s_0,s_1,\ldots,s_n$ real and \emph{non-negative}, such that $\sum_{i=1}^{n}s_i=s_0>0$, defining $B=\bigvee_{i=0}^{n} B_{i}$ and
$\lgain=\sum_{i=1}^{n}s_{i}B_{i}(X_{i}-\lpr(X_{i}|B_{i}))-s_{0}B_{0}(X_{0}-\lpr(X_{0}|B_{0}))$,
the following condition holds: $\sup(\lgain|B)\geq 0$.

$\lpr$ is \emph{centered convex (C--convex)} when it is convex and $X|B\in \setran$ implies that $0|B\in \setran$ and $\lpr(0|B) = 0$.
\end{definition}
These previsions were studied in \cite{Pel05} and
are related to certain kinds of risk measures.
Convex previsions that are not necessarily centered have been sometimes considered in the literature,
for instance in \cite{Fol02} with the corresponding concept of (unconditional) convex risk measure.
C-convex previsions ensure however definitely better consistency properties
(cf. also the discussion following Proposition \ref{pro:GN_measures}).
We shall mainly refer to them in what follows.

Precise conditional previsions may be defined similarly \cite{Hol85},
extending de Finetti's notion of coherence for unconditional previsions \cite{deF74}:
\begin{definition}
\label{def:condPrev} $P:\setran\longmapsto\reals$ is a dF-coherent
conditional prevision iff, for all $n \in\nats$, $\forall\
X_1|B_1,\ldots,X_n|B_n \in \setran$, $\forall\ s_i \in \reals$
$(i=1,\ldots,n)$, defining
$G=\sum_{i=1}^{n}s_{i}B_{i}(X_i-P(X_i|B_i))$,
$B=\bigvee_{i=1}^{n}B_i$, it holds that $\sup(G|B)\geq 0$.
\end{definition}
\begin{remark}
\label{rem:c-convexity}
C-convexity is more general than W-coherence and dF-coherence.
Hence, the properties of C-convex previsions hold for W-coherent and dF-coherent previsions too.
They apply to Walley's coherence too,
whenever it is equivalent to W-coherence.
This is the case, for instance, of properties involving finitely many events,
like Proposition \ref{pro:preserve_inequality}.
\hfill $\blacklozenge$
\end{remark}
It is well known that dF-coherent (W-coherent, C-convex) previsions on $\setran$ allow for extensions on \emph{any}
set of conditional gambles $\setran^\prime\supset\setran$ which are dF-coherent (W-coherent, C-convex, respectively).
The special extension problem in the next lemma will be needed in the proof of Proposition \ref{pro:ext_prob}.
\begin{lemma}
\label{lem:extension}
Let $\mu(\cdot|\cdot)$ be a dF-coherent prevision
(alternatively, a W-coherent or C-convex lower or upper prevision) on $\setran$.
Suppose that $\mu^\prime$,$\mu^{\prime\prime}$ are two dF-coherent
(alternatively, W-coherent or C-convex)
extensions of $\mu$ on $\setran\cup\{X|D\}$,
and that $\mu^\prime(X|D)=m$, $\mu^{\prime\prime}(X|D)=M>m$.
Then, any extension $\mu_{ext}$ of $\mu$ on $\setran\cup\{X|D\}$ such that
$\mu_{ext}(X|D)\in [m,M]$
is dF-coherent (respectively, W-coherent, C-convex). 
\end{lemma}
\begin{proof}
The consistency of $\mu_{ext}$ may be proved by checking Definitions \ref{def:W-coherence}, \ref{def:C-convexity} and \ref{def:condPrev},
or using conjugacy in the upper prevision case.
The procedure is essentially the same,
and we exemplify it when $\mu$ is a C-convex lower prevision on $\setran$.
Then, by Definition \ref{def:C-convexity},
$\mu_{ext}$ is C-convex on $\setran\cup\{X|D\}$ iff
$\forall X_0|B_0,\ldots,X_n|B_n\in\setran\cup\{X|D\}$,
$\forall s_0,\ldots,s_n$ real and non--negative,
such that $\sum_{i=1}^{n}s_{i}=s_{0}>0$,
\begin{equation}
\label{eq:gain_ext}
\sup\lgain|B=\sup\{
\sum_{i=1}^{n}s_i B_i (X_i-\mu_{ext}(X_i|B_i))
-s_0 B_0(X_0-\mu_{ext}(X_0|B_0))|B\}
\geq 0.
\end{equation}
Clearly,
there is nothing to prove if $X_0|B_0,\ldots,X_n|B_n\in S$.
If not,
let us call $s$ the stake regarding $X|D$.\footnote{Since $X_0|B_0,\ldots,X_n|B_n$ need not be distinct,
if $X|D$ is present more than once $s$ is the stake of the sum of the terms where it appears.}
Then, we write
$\lgain|B=sD(X-\mu_{ext}(X|D))+R|B$,
where $R$ consists of the remaining terms in $\lgain$.

If $s\geq 0$, $\lgain|B\geq sD(X-M)+R|B=\lgain^{\prime\prime}|B$.
Therefore, $\sup\lgain|B\geq\sup\lgain^{\prime\prime}|B\geq 0$,
the last inequality holding because the gain $\lgain^{\prime\prime}|B$ concerns the extension $\mu^{\prime\prime}$ of $\mu$.

If $s<0$, $\lgain|B\geq sD(X-m)+R|B=\lgain^\prime|B$.
Since $\lgain^\prime|B$ is a gain regarding $\mu^\prime$,
the conclusion is the same.
\end{proof}

Among the W-coherent extensions on $\setran^\prime\supset\setran$ of a lower prevision $\lpr$,
the \emph{natural extension} $\lnex$ is its least-committal one.
This means that for any $\underline{Q}$ such that $\underline{Q}$ is W-coherent on $\setran^\prime$ and $\underline{Q}=\lpr$ on $\setran$,
it holds that $\underline{Q}\geq\lnex$ ($\lnex$ is dominated by $\underline{Q}$).
The concept of C-convex natural extension is analogous for C-convex previsions.
The natural extension (the C-convex natural extension) always exists and is unique \cite{Pel05,Pel09, Wil07}.

Sometimes one may be interested in searching for an extension $\underline{U}$ of $\lpr$ with opposite features,
i.e. ensuring that no (W-coherent, alternatively C-convex) extension $\underline{Q}$ of $\lpr$ is such that $\underline{Q}\geq\underline{U}$.
This is the notion of \emph{upper extension},
originally developed in \cite{Wei01}.
The upper extension is generally not unique,
and its practical computation may be not immediate.
We shall meet a case of upper extension in Section \ref{subsec:natural_extension}.

A property of W-coherent previsions is the \emph{(weak) product rule}, proven in \cite{Pel09bis}:
\begin{proposition}
\label{pro:product_rule}
Let $\lpr$ be W-coherent on $\mathcal{S}\supset\{AX|B, A|B, X|A\wedge B\}$.
Then, necessarily:
\begin{itemize}
\item[a)] if $\lpr(X|A\wedge B)>0$, then
\begin{eqnarray}
\label{eq:product_rule_1}
\lpr(AX|B)\geq\lpr(A|B)\cdot\lpr(X|A\wedge B)
\end{eqnarray}
\item[b)] if $\lpr(X|A\wedge B)<0$, then
\begin{eqnarray}
\label{eq:product_rule_2}
\lpr(AX|B)\leq\lpr(A|B)\cdot\lpr(X|A\wedge B)
\end{eqnarray}
\item[c)]
$\lpr(AX|B)=0\ \mbox{iff}\ \lpr(A|B)\cdot\lpr(X|A\wedge B)=0$
\end{itemize}
\end{proposition}
In the consistency concepts above,
the measure $\mu$ is a (lower, upper or precise) \emph{probability} if,
for any $X|B\in \setran$, $X$ is (the indicator of) an event.
In all such cases,
the following are necessary consistency conditions:
\begin{equation}
\label{eq:consistency_conditions}
\mu(A|B)\in [0;1],\ \mu(\varnothing|B)=0,\ \mu(B|B)=1.
\end{equation}
In general, results for upper probabilities follow from those for lower probabilities by the conjugacy equality
$\upr(A|B)=1-\lpr(\nega{A}|B)$.

\subsection{The Goodman-Nguyen Relation}
\label{subsec:GN_relation}
\begin{definition}
\label{def:goongu} \emph{(Goodman--Nguyen relation.)} We say that
$A|B\lgn C|D$ iff
\begin{equation}
\label{eq:goongu}
A\wedge B\Rightarrow C\wedge D \mbox{ and } \nega{C}\wedge D\Rightarrow\nega{A}\wedge B.
\end{equation}
\end{definition}
\begin{example}
\label{ex:impl_chain}
Some simple examples of GN-related events:
\begin{itemize}
\item[a)] If $A\Rightarrow C\Rightarrow D\Rightarrow B$,
then $A|B\lgn C|D$;
\item[b)] $A|B\lgn C|B$ iff $A\wedge B \Rightarrow C\wedge B$;
\item[c)] $\varnothing|B\lgn C|D$ iff $\nega{C}\wedge D\Rightarrow B$.
\end{itemize}
\end{example}
The GN relation was apparently first introduced by Goodman and Nguyen in \cite{Goo88}.
In that paper,
the focus is on defining the operations $\wedge$, $\vee$ with conditional events,
which is done as follows:
\begin{equation}
\label{eq:oper_cond_events}
\begin{aligned}
A|B\wedge C|D&=(A\wedge B\wedge C\wedge D)|[(\nega{A}\wedge B)\vee(\nega{C}\wedge D)\vee (B\wedge D)]\\
A|B\vee C|D&=[(A\wedge B)\vee (C\wedge D)]|[(A\wedge B)\vee (C\wedge D)\vee (B\wedge D)]
\end{aligned}
\end{equation}
The relation $\lgn$ is then defined as
\begin{equation}
\label{eq:GN_original}
A|B\lgn C|D \text{ iff } A|B=A|B\wedge C|D \text{ iff } C|D=C|D\vee A|B
\end{equation}
and it is stated without proof that this definition is equivalent to Definition~\ref{def:goongu}.
The equivalence is discussed at length in \cite{Mil97},
where it is also asserted that equations \eqref{eq:oper_cond_events} can already be deduced from the truth tables presented by de~Finetti in \cite{deF36}.
While de Finetti does not seem to have considered explicitly the GN relation,
the intuition behind Definition \ref{def:goongu} has been explained in the literature resorting to betting arguments,
much in his style
(cf. \cite{Mil97} and also \cite{Mil08},
where $\lgn$ is termed `betting order').

In fact, \eqref{eq:goongu} states that whenever we bet both on $A|B$
and on $C|D$ (iff $B\wedge D$ is true), the following holds:
if we win the bet on $A|B$, we also win the bet on $C|D$ (because $A\wedge B\Rightarrow C\wedge D$);
and if we loose the bet on $C|D$, we also loose the bet on
$A|B$ (because of $\nega{C}\wedge D\Rightarrow\nega{A}\wedge B$).
When $B=D=\Omega$, just one of the implications in \eqref{eq:goongu} is
needed, because of the tautology $A\Rightarrow C\leftrightarrow
\nega{C}\Rightarrow\nega{A}$.
We shall reconsider and broaden the betting interpretation in Section 
\ref{subsec:GN relation conditional events}.

The GN relation was a secondary item,
at best,
in Goodman and Nguyen's work on \emph{conditional event algebras}.
Recently \cite{Gil13b}, it has been related,
together with Adam's quasi-conjunction,
to probabilistic entailment under coherence.

In this paper, we do not follow these lines of research,
but investigate rather the relevance of $\lgn$ in imprecise probability theory.
In fact,
the recalled interpretation of $\lgn$ suggests that \eqref{eq:defleGN} should hold
for a generic, but consistent uncertainty measure $\mu$.
In the case that $\mu$ is a conditional probability $P$, \eqref{eq:defleGN}
was stated without proof in \cite{Goo88}
and proven in \cite{Goo91}
(assuming $P$ defined on
a structured set, termed $\strset$ in Definition \ref{def:inner_outer}).
Equation \eqref{eq:defleGN} was proved also in \cite{Col93} (under general assumptions for $P$)
and independently (in a less general case) in \cite{Mil97}.
Propositions \ref{pro:preserve_inequality} and \ref{pro:GN_measures} in this paper establish \eqref{eq:defleGN} for imprecise measures.

The GN relation in extension problems is explored in
\cite{Col96} in the context of precise probabilities.
It is shown there that, given a dF-coherent probability $P$ on a
\emph{finite} set of conditional events, the bounds for its coherent
extensions on one additional event $C|D$ depend on the values of $P$
on two events.
The two events are determined by the GN relation and
are termed $(C|D)_*$ and $(C|D)^*$ in equation \eqref{eq:max_cond_event}.

In Section \ref{subsec:natural_extension} we study more general extension problems, on \emph{arbitrary} sets
of events and for precise or imprecise probability assessments.


\section{Effects of the GN Relation on Imprecise Probability Assessments}
\label{sec:GN_relation_ip}
The GN relation induces a corresponding ordering, satisfying equation \eqref{eq:defleGN},
on C-convex lower/upper conditional probabilities,
and therefore (Remark \ref{rem:c-convexity}) on W- and Walley-coherent imprecise probabilities,
and on dF-coherent probabilities:
\begin{proposition}
\label{pro:preserve_inequality}
Let $\mu$ be a C-convex lower (or upper) probability defined on $\mathcal{D}\supseteq\left\lbrace A|B, C|D\right\rbrace$.
Then, $A|B\lgn C|D$ implies $\mu(A|B)\leq\mu(C|D)$.
\end{proposition}
Proposition \ref{pro:preserve_inequality} was proved in \cite{Pel13}.
It is a special case of Proposition \ref{pro:GN_measures},
proved in Section \ref{subsec:Ordering from GN relation}.
\begin{example}
\label{ex:two_evaluations}
In several common situations,
we may be interested in evaluating an event $A$ conditioned on different,
but increasingly more precise assumptions.
This originates a sequence of conditional events $A|B_{n}$,
with the conditioning events $B_{n}$, $n=0,1,\ldots$, totally ordered by implication.

Does this special structure imply some ordering,
according to the GN relation, among the conditional events $A|B_{n}$?
To answer this question,
let us compare $A|B_{0}$ with $A|B_{1}$,
assuming that $B_{1}\Rightarrow B_{0}$.
 
There are three possible situations:
\begin{itemize}
\item[a)] $A|B_{0}\lgn A|B_{1}$
iff $A\wedge B_{0}\wedge\nega{B_1}=\varnothing$.
 
To see this,
apply Definition \ref{def:goongu}:
 $A|B_{0}\lgn A|B_{1}$ iff
 $A\wedge B_{0}\Rightarrow A\wedge B_{1}$ and $\nega{A}\wedge B_{1}\Rightarrow\nega{A}\wedge B_{0}$.
 Since $B_{1}\Rightarrow B_{0}$,
 we get
 $A|B_{0}\lgn A|B_{1}$ iff $A\wedge B_{0}\Rightarrow A\wedge B_{1}$
 iff $\nega{A}\wedge B_{0}\wedge\nega{B_{1}}=\varnothing$.
 %
 %
 \item[b)] $A|B_{1}\lgn A|B_{0}$
 iff $\nega{A}\wedge B_{0}\wedge\nega{B_{1}}=\varnothing$.
 
 The proof of b) is in line with that of a),
 noting that this time the first implication in Definition \ref{def:goongu} is always true.
 \item[c)] Neither $A|B_{0}\lgn A|B_{1}$
 nor $A|B_{1}\lgn A|B_{0}$ iff
 ($A\wedge B_{0}\wedge\nega{B_1}\neq\varnothing$ and
 $\nega{A}\wedge B_{0}\nega{B_{1}}\neq\varnothing$).
 \end{itemize}
 Let now $\mu$ be a C-convex probability,
 either lower or upper.
 In cases a) and b), Proposition \ref{pro:preserve_inequality} lets us derive an inequality linking
 $\mu(A|B_{0})$ and $\mu(A|B_{1})$.
 For instance,
 with case a) we get
 \begin{equation}
 \label{eq:increasing_inequality}
 A\wedge B_{0}\wedge\nega{B_{1}}=\varnothing\rightarrow\mu( A|B_{0})\leq\mu( A|B_{1}).
 \end{equation}
The inequality in \eqref{eq:increasing_inequality} is already known in some special cases.
 In particular,
 let $A\Rightarrow B_{1}\Rightarrow B_{0}$ (see Figure \ref{fig:inclusion}),
 \begin{figure}[h]
 \centering
 \includegraphics[width=0.5\textwidth]{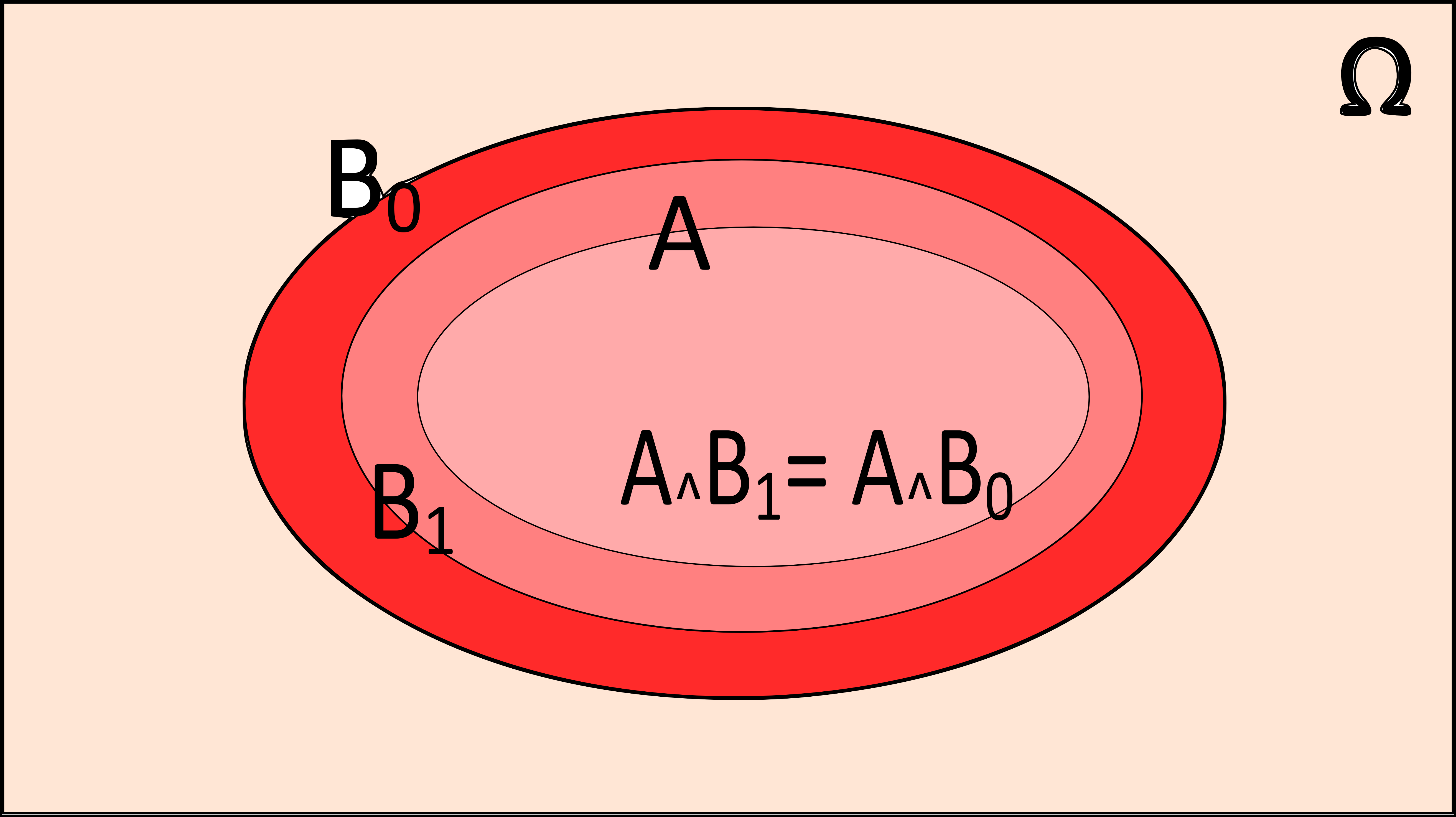}
 \caption{Example 2, a).}
 \label{fig:inclusion}
 \end{figure}
 which implies that $A\wedge B_{0}\wedge \nega{B_{1}}=\varnothing$.
 Then the inequality follows from the product rule
 $P(A|B_0)=P(A|B_1)P(B_1|B_0)$
 when $\mu$ is a dF-coherent probability,
 while it was obtained in a different way in \cite{Pel10} for W-coherent probabilities.
\end{example}
\begin{remark}
\label{rem:incomparable}
\emph{(GN-unrelated events.)}
As appears from Example \ref{ex:impl_chain}, c) and Example \ref{ex:two_evaluations}, c),
several conditional events are not GN-related.
It may be so even in the presence of implication relationships between either their conditioned events
($\varnothing$ implies any event in Example \ref{ex:impl_chain}, c)),
or their conditioning ones
($B_{0}\Rightarrow B_{1}$ in Example \ref{ex:two_evaluations}, c)).

Further,
since $B_{1}\Rightarrow B_{0}$ induces $\mu(B_{1})\leq\mu(B_{0})$
for any monotone measure $\mu$,
one might suspect that the agreeing ordering $\mu(A|B_{1})\leq\mu(A|B_{0})$
should hold with GN-related $A|B_{1}$ and $A|B_{0}$.
We have seen instead in Example \ref{ex:two_evaluations}, b)
that this is not always the case.
Yet,
the following holds:
if $A$ is arbitrary, while $B_{1}\Rightarrow B_{0}$,
then
\begin{equation}
\label{eq:mu_oinequality}
\mu(A\wedge B_1|B_0)\leq\mu(A|B_1).
\end{equation}
Inequality \eqref{eq:mu_oinequality} follows from Proposition \ref{pro:preserve_inequality},
since $A\wedge B_{1}|B_{0}\lgn A\wedge B_{1}|B_{1}=A|B_{1}$.
\hfill $\blacklozenge$ 
\end{remark}

\begin{remark}
\emph{(Relationship with conditional implication.)}
The GN relation is linked to conditional implication as follows.
Recall that the conditional implication $A|H\Rightarrow B|H$ may be defined as
$A\wedge H\Rightarrow B\wedge H$
(or alternatively, by the truth table of $A\Rightarrow B$, provided that $H$ is true).

Suppose then $A|B\lgn C|D$.
From $(A\wedge B\Rightarrow C\wedge D)\rightarrow[(A\wedge B)\wedge(B\wedge D)\Rightarrow(C\wedge D)\wedge(B\wedge D)]
\leftrightarrow(A\wedge B\wedge D\Rightarrow C\wedge B\wedge D)
\leftrightarrow(A|B\wedge D\Rightarrow C|B\wedge D)$,
we get that
\begin{equation}
\label{eq:conditional_implication}
A|B\lgn C|D\rightarrow (A|B\wedge D\Rightarrow C|B\wedge D).
\end{equation}
Similar computations using $\nega{C}\wedge D\Rightarrow\nega{A}\wedge B$ show that
\begin{equation}
\label{eq:conditional_implication_nega}
A|B\lgn C|D\rightarrow (\nega{C}|B\wedge D\Rightarrow \nega{A}|B\wedge D).
\end{equation}
Hence, the GN relation implies the two conditional implications in \eqref{eq:conditional_implication} and \eqref{eq:conditional_implication_nega}.
As a follow up,
note that, for a dF-coherent $P$
\begin{equation}
\label{eq:conditional_implication_precise}
A|B\lgn C|D\rightarrow P(A\wedge D|B)\leq P(C\wedge D|B).
\end{equation}
In fact, from \eqref{eq:conditional_implication}, $A|B\wedge D\Rightarrow C|B\wedge D$ ensures $P(A|B\wedge D)\leq P(C|B\wedge D)$.
Multiplying both terms by $P(D|B)$ gives the inequality in \eqref{eq:conditional_implication_precise}.

The conditional implications in \eqref{eq:conditional_implication}, \eqref{eq:conditional_implication_nega} jointly have
the same betting interpretation recalled in Section \ref{subsec:GN_relation} for the GN relation.
However,
the GN relation may compare events with \emph{different} conditioning events.
This interesting differentiating feature is useful,
for instance,
in the problems of the next section.
\hfill $\blacklozenge$
\end{remark}

\subsection{The GN Relation in Extension Problems}
\label{subsec:natural_extension}
We shall discuss now generalisation of the extension problem presented in the Introduction.
\begin{definition}
\label{def:inner_outer}
Let $\prt$ be any partition.
Given an event $E$,
define its \emph{inner event}
$E_*=\bigvee\limits_{e \in\prt:\ e\Rightarrow E} e$
and its \emph{outer event}
$E^*=\bigvee\limits_{e\in\prt:\ e\wedge E\neq\varnothing} e$.
Define further $\strset=\strset(\prt)=\{A|B:A,B\in\genfield,B\neq\varnothing\}$,
and for an arbitrary $C|D$ ($C|D\neq\varnothing|D$, $C|D\neq D|D$),
$m(C|D)=\{A|B\in\strset(\prt): A|B\lgn C|D\}$,
$M(C|D)=\{A|B\in\strset(\prt): C|D\lgn A|B\}$.
\end{definition}
Note that the definitions of inner and outer event are not independent:
$E^*=\nega{((\nega{E})_*)}$.
It is easy to see that
\begin{proposition}
\label{pro:set_asterisks} The sets $m$, $M$ are non-empty and have,
respectively, a maximum $(C|D)_*$ and a minimum $(C|D)^*$
conditional event w.r.t. $\lgn$,
\begin{equation}
\label{eq:max_cond_event}
\begin{array}{lll}
(C|D)_*=(C\wedge D)_*|[(C\wedge D)_*\vee(\nega{C}\wedge D)^*],\\
(C|D)^*=(C\wedge D)^*|[(C\wedge D)^*\vee(\nega{C}\wedge D)_*].
\end{array}
\end{equation}
\end{proposition}
Analogously to the unconditional case,
$(C|D)_*$ may be termed the \emph{inner} event and $(C|D)^*$ the \emph{outer} event of $C|D$.
Both are made up of unconditional inner and outer events, by Definition \ref{def:inner_outer}.
For instance,
$(C\wedge D)_*=\bigvee\limits_{e\in\prt:\ e\Rightarrow C\wedge D} e$,
$(\nega{C}\wedge D)^*=\bigvee\limits_{e\in\prt:\ e\wedge\nega{C}\wedge D\neq\varnothing} e$.

A graphical illustration of $(C|D)_*$ is supplied in Figure \ref{fig:conditional_starred}.
\begin{figure}[h]
\centering
\includegraphics[width=0.5\textwidth]{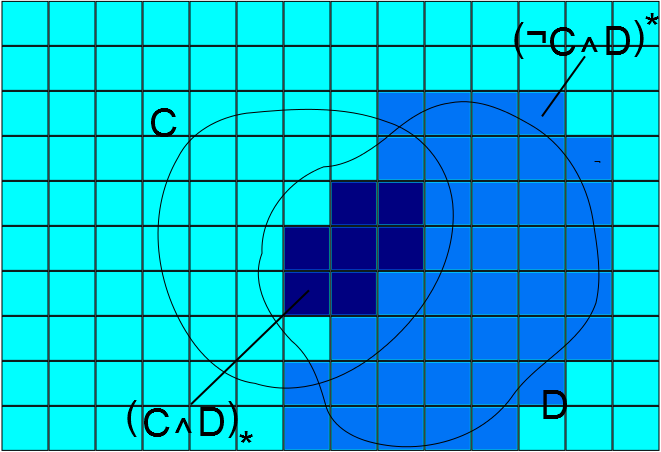}
\caption{The unconditional events forming $(C|D)_*$.}
\label{fig:conditional_starred}
\end{figure}

Suppose now that an uncertainty measure $\mu$
is assessed on the set of conditional events $\strset(\prt)$ (Definition \ref{def:inner_outer}).
We wish to extend $\mu$ to an arbitrary event $C|D$.
The non-triviality assumption $C|D\neq\varnothing|D$, $C|D\neq D|D$, already introduced in Definition \ref{def:inner_outer},
is assumed also in the sequel.
It rules out limiting cases whose extension is already known by \eqref{eq:consistency_conditions}.

It holds that
\begin{proposition}
\label{pro:ext_prob}
Let $\mu(\cdot|\cdot)$ be a dF-coherent probability,
or alternatively a W-coherent or C-convex lower or upper probability on $\strset(\prt)$.
Any of its extensions on $\strset(\prt)\cup\{C|D\}$ is, respectively, a dF-coherent probability, or a W-coherent or C-convex lower or upper probability
if and only if $\mu(C|D)\in [\mu((C|D)_*); \mu((C|D)^*)]$.
\end{proposition}
\begin{proof}
The \emph{only if} part follows at once from Proposition \ref{pro:preserve_inequality}:
$(C|D)_{*}\lgn C|D\lgn (C|D)^{*}$ implies
$\mu((C|D)_{*})\leq\mu(C|D)\leq\mu((C|D)^{*})$.

The proof of the \emph{if} implication consists of two parts:
\begin{itemize}
\item[a)] Prove that both $\mu(C|D)=\mu((C|D)_*)$ and $\mu(C|D)=\mu((C|D)^*)$ are consistent extensions of $\mu$.
\item[b)] Apply Lemma \ref{lem:extension}.
\end{itemize}
To prove a), define $\mu(C|D)=\mu((C|D)_*)$ and let $G|K$ be a generic gain
concerning $\mu$ on $\strset(\prt)\cup\{C|D\}$.
We shall check that the maximum of $G|K$ is non-negative.

The gain $G|K$ involves bets regarding $C|D$ and (no or) finitely many events of $\strset(\prt)$,
call them $A_1|B_1,\ldots,A_n|B_n$ for notational simplicity.
Here $n\geq 0$,
whilst $C|D$ is necessarily included into the bets.
If not, we are left with a bet on $\strset(\prt)$ only.
This would immediately imply $\max G|K\geq 0$,
since $\mu$ is already known to be consistent on $\strset(\prt)$.

Defining
$G_0=\sum_{i=1}^{n}s_i B_i(A_i-\mu(A_i|B_i))$,
we may write
\begin{equation*}
G=sD(C-\mu((C|D)_*))+G_0,\ K=D\vee\bigvee_{i=1}^{n}B_i.
\end{equation*}
The term $G_0$ gathers the addends in $G$ referring to bets regarding events in $\strset(\prt)$.

We shall consider for each $G|K$ an auxiliary gain $G_{aux}|K_{aux}$,
such that
\begin{itemize}
\item[i)] $\max G_{aux}|K_{aux}\geq 0$.
\item[ii)] The set of values that $G_{aux}|K_{aux}$ may take is a subset of the possible values of $G|K$.
\end{itemize}
Clearly, i) and ii) jointly imply $\max G|K\geq 0$, hence the first part of a).

The auxiliary gain $G_{aux}|K_{aux}$ is defined as follows:
\begin{equation*}
G_{aux}=sD^m(C^m-\mu((C|D)_*))+G_{0},\ K_{aux}=D^{m}\vee\bigvee_{i=1}^{n}B_i
\end{equation*}
where
\begin{equation}
\begin{split}
\label{eq:cmdm}
C^{m}&=(C\wedge D)_*,\\
D^{m}&=(C\wedge D)_{*}\vee (\nega{C}\wedge D)^{*}.
\end{split}
\end{equation}
In equation \eqref{eq:cmdm},
we simply redefine the events forming $(C|D)_{*}$.
Therefore, $C^{m}|D^{m}=(C|D)_{*}$.

Note that $G$ and hence $G_{aux}$ may be gains for a dF-coherent, W-coherent or C-convex upper or lower $\mu$
by putting suitable additional constraints on $s,s_1,\ldots,s_n$.
Our proof is independent of which constraints are possibly added,
hence holds in all the above assumptions for $\mu$.

\emph{Proof of i)}.
The inequality
$\max G_{aux}|K_{aux}\geq 0$ follows from the assumed consistency (dF-coherence, etc.) of $\mu$ on $\strset(\prt)$.
In fact, all the events in $G_{aux}|K_{aux}$ belong to $\strset(\prt)$.

\emph{Proof of ii)}.
We preliminarily note that
$G_{aux}|K_{aux}$ is defined on the elements of $\prt$ implying $K_{aux}$,
whilst an appropriate partition to evaluate $G|K$ is the product partition $\prt^\prime=\prt\wedge\{C\wedge D, \nega{C}\wedge D, \nega{D}\}$.
Clearly, $\prt^\prime$ is finer than $\prt$.

We prove that $\forall e \in\prt$ such that $e \Rightarrow K_{aux}$
there exists $\omega\in\prt^{\prime}$ such that $\omega\Rightarrow K$
and $G(\omega)=G_{aux}(e)$,
i.e. $G|K(\omega)=G_{aux}|K_{aux}(e)$.

It is useful for this to observe that any atom $e$ of $\prt$ can be written
as the logical sum of at most three non-impossible atoms of $\prt^\prime$:
\begin{equation}
\label{eq:atomsplit}
e=(e\wedge C\wedge D)\vee(e\wedge\nega{C}\wedge D)\vee(e\wedge \nega{D}).
\end{equation}
Writing $K_{aux}$ as the sum of two disjoint events,
$K_{aux}=D^{m}\vee(\bigvee_{i=1}^{n}B_i\wedge\nega{D^{m}})$,
we correspondingly consider two cases:
\begin{itemize}
\item[a1)] Let $e\in\prt$ be such that $e\Rightarrow\bigvee_{i=1}^{n}B_i\wedge\nega{D^{m}}\ (\Rightarrow\nega{D^{m}})$.

Then $e\wedge D^{m}=\varnothing$ and,
therefore, also $e\wedge\nega{C}\wedge D=\varnothing$.
Hence, \eqref{eq:atomsplit} implies $e=(e\wedge C\wedge D)\vee (e\wedge\nega{D})$.
Note that $\omega=e\wedge\nega{D}$ is non-impossible,
because $\omega=\varnothing$ would imply
$e=e\wedge C\wedge D\Rightarrow (C\wedge D)_{*}\Rightarrow D^{m}$,
which is a contradiction to the assumption.
Furthermore,
$\omega\Rightarrow e\Rightarrow\bigvee_{i=1}^{n} B_{i}\Rightarrow K$.
It is now immediate to check that $G_{aux}(e)=G(\omega)(=G_{0}(e))$ and,
therefore, $G_{aux}|K_{aux}(e)=G|K(\omega)$.
\item[a2)] Let $e\in\prt$ be such that $e\Rightarrow D^{m}=(C\wedge D)_{*}\vee (\nega{C}\wedge D)^{*}$.

If $e\Rightarrow (C\wedge D)_*=C^{m}$,
then $e\Rightarrow C\wedge D$.
Hence, $e=e\wedge C\wedge D$ in \eqref{eq:atomsplit}.
Consequently, $e$ is also an element of $\prt^\prime$ and $e\Rightarrow K$.
This means that also $G|K$ is defined at $e$ and,
since $G_{aux}(e)=G(e)=s(1-\mu((C|D)_*))+G_{0}(e)$,
we have that $G_{aux}|K_{aux}(e)=G|K(e)$.

If $e\Rightarrow(\nega{C}\wedge D)^*$,
then at least the component $\omega^\prime=e\wedge\nega{C}\wedge D$ of $e$ in equation \eqref{eq:atomsplit}
is non-impossible and implies $K$.
Then $G_{aux}(e)=G(\omega^\prime)=-s\mu((C|D)_*)+G_{0}(e)$.
Therefore, $G_{aux}|K_{aux}(e)=G|K(\omega^\prime)$.
\end{itemize}
Thus we have proven ii).

The proof of the second part of a) is quite analogous.
Putting now $\mu(C|D)=\mu((C|D)^*)$,
we consider a generic gain
$G^\prime|K$ concerning $\mu$ on  $\strset\cup\{C|D\}$,
with $G^\prime=sD(C-\mu((C|D)^*))+G_0$.
We define further an auxiliary gain $G^{aux}|K^{aux}$,
with $G^{aux}=sD^{M}(C^{M}-\mu((C|D)^*))+G_1$,
$K^{aux}=D^{M}\vee\bigvee_{i=1}^{n}B_i$,
where $C^{M}=(C\wedge D)^{*}$,
$D^{M}=(C\wedge D)^{*}\vee (\nega{C}\wedge D)_{*}$.
Then, the analogues of i), ii) can be proven.
For ii), two major alternatives are to be considered,
as suggested by the decomposition $K^{aux}=D^{M}\vee(\bigvee_{i=1}^{n}B_i\wedge\nega{D^{M}})$.
We omit the details.

Finally, the thesis of the Proposition follows by b).
\end{proof}

In the general case of extensions on an \emph{arbitrary} set of conditional events,
the following proposition holds:
\begin{proposition}
\label{pro:ext_prob_many}
Let $\lpr$ ($\upr$) be a W-coherent, respectively C-convex lower (upper) probability defined on $\strset$.
Let also $\seteve$ be an arbitrary set of conditional events.
Then, the extension of $\lpr$ ($\upr$) on $\strset\cup\seteve$,
such that $\lpr(C|D)=\lpr((C|D)_*)$ ($\upr(C|D)=\upr((C|D)^*)$), $\forall C|D\in\seteve$,
is a W-coherent, respectively C-convex lower (upper) probability.
\end{proposition}
\begin{proof}
We prove the part of the thesis concerning $\lpr$,
the one regarding $\upr$ being quite analogous.

\begin{itemize}
\item[a)] We start with a preliminary fact:
let $C|D\in\seteve$, $s\geq 0$,
and define
\begin{equation}
g=sD(C-\lpr(C^m|D^m)),\
g^m=sD^m(C^m-\lpr(C^m|D^m)),
\end{equation}
where $C^{m}$ and $D^{m}$ are defined in \eqref{eq:cmdm}.

Then $g\geq g^m$.

To prove this,
note firstly that both $g$ and $g^m$ are defined on the product partition
$\{C\wedge D, \nega{C}\wedge D, \nega{D}\}\wedge\{C^m\wedge D^m, \nega{{C^m}}\wedge D^m, \nega{{D^m}}\}$.
Three elements of this partition are impossible,
namely $e_1=C^m\wedge D^m\wedge\nega{C}\wedge D$, $e_2=C^m\wedge D^m\wedge\nega{D}$, $e_3=\nega{{D^m}}\wedge\nega{C}\wedge D$.
In fact,
$C^m\wedge D^m=C^m=(C\wedge D)_{*}\Rightarrow C\wedge D$ by the definition of $C^m$,
hence $e_1=e_2=\varnothing$.
Also, $e_3=\varnothing$: from the definition of $D^m$, $\nega{C}\wedge D\Rightarrow D^{m}$,
that is $\nega{C}\wedge D \wedge \nega{D^{m}}=\varnothing$.
The inequality $g\geq g^m$ follows  
by comparing the values of $g$ and $g^m$ at the remaining $6$ atoms.
For instance,
at $e_{4}=\nega{{C^m}}\wedge D^m\wedge C\wedge D$,
$g(e_{4})=s(1-\lpr(C^m|D^m))>-s\lpr(C^m|D^m)=g^m(e_{4})$.
The other cases are similar.
\item[b)] Now take a generic gain $\lgain$ in the definition of W-coherence or C-convexity of the extension of $\lpr$ on $\strset\cup\seteve$,
\begin{equation}
\label{eq:gain_main}
\lgain=g_0+\sum_{h=1}^{r}g_h+\sum_{k=1}^{n}t_k B_k(A_k-\lpr(A_k|B_k)).
\end{equation}
Here $g_h=s_h D_h(C_h-\lpr((C_h|D_h)_{*}))$ and $C_h|D_h\in\seteve$ $(h=0,\ldots,r)$,
whilst $A_k|B_k\in\strset$ $(k=1,\ldots,n)$.

Without introducing any real restriction,
we may suppose that the only stake that can (with W-coherence) or must (with C-convexity) be \emph{negative} is either $s_0$ or one among $t_1,t_2,\ldots,t_n$.
The remaining stakes are necessarily non-negative.

Next to this,
we compare any $\lgain$ with its `auxiliary' gain
\begin{equation}
\label{eq:gain_aux}
\lgain_{aux}=g_0+\sum_{h=1}^{r}g_h^m+\sum_{k=1}^{n}t_k B_k(A_k-\lpr(A_k|B_k)),
\end{equation}
where $g_h^m=s_h D_h^m(C_h^m-\lpr(C_h^m|D_h^m))$.
The events $C_h^m$, $D_h^m$ are derived from $C_h$, $D_h$ $(h=1,\ldots,r)$,
exactly like $C^m$, $D^m$ from $C$, $D$ in \eqref{eq:cmdm}.
Hence, $C_{h}^{m}|D_{h}^{m}=(C_{h}|D_{h})_{*}$.

$\lgain$ and $\lgain_{aux}$ are defined on the product partition
$\prt_{aux}=\prt\wedge\bigwedge_{h=o}^{r}\{C_h\wedge D_h,\nega{{C_h}}\wedge D_h,\nega{{D_h}}\}$,
and we easily get, by a),
\begin{equation}
\label{eq:non_aux_dominates_aux_gain}
\lgain\geq\lgain_{aux}.
\end{equation}
\item[c)]
Define
$K=D_0\vee\bigvee_{h=1}^{r}D_h\vee\bigvee_{k=1}^{n}B_k$,
$K_{aux}=D_0\vee\bigvee_{h=1}^{r}D_h^m\vee\bigvee_{k=1}^{n}B_k$.
We are now going to prove that $\max\lgain|K\geq0$.

Note that $\max\lgain_{aux}|K_{aux}\geq0$ by Proposition \ref{pro:ext_prob}
and Definition \ref{def:W-coherence} (W-coherence) or Definition \ref{def:C-convexity} (C-convexity).
In fact, all events in $\lgain_{aux}$ except $C_0|D_0$ belong to $\strset(\prt)$.

We define further the events
$K_{aux}^{1}=D_0\vee\bigvee_{h=1}^{r}(D_h^m\wedge D_h)\vee\bigvee_{k=1}^{n}B_k$,
$K_{aux}^{2}=\bigvee_{h=1}^{r}(D_h^m\wedge \nega{(D_h)})\wedge \nega{(K_{aux}^{1})}$.
They are clearly disjoint and such that $K_{aux}=K_{aux}^{1}\vee K_{aux}^{2}$.

Recall that $\max\lgain_{aux}|K_{aux}=\max\{\max\lgain_{aux}|K_{aux}^{1},\max\lgain_{aux}|K_{aux}^{2}\}$$\geq 0$.
From this, we distinguish two cases to evaluate
$\max\lgain |K$. 
\begin{itemize}
\item[c1)]
$\max\lgain_{aux}|K_{aux}=\max\lgain_{aux}|K_{aux}^1$.

Since $K_{aux}^{1}\Rightarrow K$
and using \eqref{eq:non_aux_dominates_aux_gain},
we get
$\max\lgain|K\geq\max\lgain|K_{aux}^{1}\geq\max\lgain_{aux}|K_{aux}^{1}=\max\lgain_{aux}|K_{aux}\geq 0$.
\item[c2)]
$\max\lgain_{aux}|K_{aux}=\max\lgain_{aux}|K_{aux}^2$.

Since $C_h^m=(C_{h}\wedge D_{h})_{*}\Rightarrow C_{h}\wedge D_h\Rightarrow D_h$ and $C_h^m\Rightarrow D_h^m$,
we easily get $C_h^m\Rightarrow D_h^m\wedge D_h\Rightarrow K_{aux}^{1}$.
Hence, $C_h^m\wedge K_{aux}^{2}=\varnothing$,
because $K_{aux}^{1}\wedge K_{aux}^{2}=\varnothing$.
It follows $\lgain_{aux}|K_{aux}^{2}=-\sum_{h=1}^{r}s_{h}\lpr(C_h^m|D_h^m)\leq 0$,
since $s_{h}\geq 0\ (h=1,\ldots,r)$.
Hence,
necessarily $\max\lgain_{aux}|K_{aux}=\max\lgain_{aux}|K_{aux}^{2}=0$,
meaning that
\begin{equation}
\label{eq:coeff_null}
s_h\lpr(C_h^m|D_h^m)=s_h\lpr((C_h|D_h)_*)=0\ \forall h=1,\ldots,r.
\end{equation} 
Define $\lgain_{aux}^{1}=g_0+\sum_{k=1}^{n}t_k B_k(A_k-\lpr(A_k|B_k))$,
$K_{aux}^{3}=D_{0}\vee\bigvee_{k=1}^{n} B_k$.
Using \eqref{eq:gain_main} and \eqref{eq:coeff_null} at the following equality,
we obtain $\lgain|K_{aux}^{3}=(\lgain_{aux}^{1}+\sum_{h=1}^{r}s_h D_h C_h)|K_{aux}^{3}\geq\lgain_{aux}^{1}|K_{aux}^{3}$.
Since $K_{aux}^{3}\Rightarrow K$,
we get
$\max\lgain|K\geq\max\lgain|K_{aux}^{3}\geq\max\lgain_{aux}^{1}|K_{aux}^{3}\geq 0$.
The last inequality holds by Proposition \ref{pro:ext_prob},
since all lower probabilities in $\lgain_{aux}^{1}|K_{aux}^{3}$ are defined in $\strset$,
except one ($\lpr(C_{0}|D_{0})=\lpr((C_{0}|D_{0})_{*})$ in $g_{0}$).
\end{itemize}
In both cases, $\max\lgain|K\geq 0$. The thesis follows.
\end{itemize}
\end{proof}
In Propositions \ref{pro:ext_prob} and  \ref{pro:ext_prob_many},
two kinds of extensions are introduced by means of the GN relation:
the \emph{lower GN-extension} $\mu(C|D)=\mu((C|D)_*)$ and
the \emph{upper GN-extension} $\mu(C|D)=\mu((C|D)^*)$, $\forall C|D\in\seteve$.
They correspond to important special extensions mentioned in Section \ref{subsec:imprecise_previsions}.
In fact
\begin{proposition}
\label{pro:GN-extensions}
Let $\lpr$ ($\upr$) be a W-coherent,
alternatively C-convex lower (upper) probability on $\strset$. Then
its lower GN-extension (its upper GN-extension)
$\lpr(C|D)=\lpr((C|D)_*)$ ($\upr(C|D)=\upr((C|D)^*)$), $\forall
C|D\in\seteve$ is the \emph{natural extension}, alternatively the
\emph{convex natural extension}, of $\lpr$ (of $\upr$) on
$\strset\cup\seteve$.
\end{proposition}
\begin{proof}
It is enough to consider the lower probability case.
Let $\underline{Q}$ be any extension of $\lpr$ on  $\strset\cup\seteve$
which is W-coherent or C-convex if $\lpr$ is so.
Then $\underline{Q}(C|D)\geq\underline{Q}((C|D)_*)=\lpr((C|D)_*)=\lpr(C|D)$,
the inequality being ensured by Proposition  \ref{pro:preserve_inequality},
the first equality because $\underline{Q}=\lpr$ in $\strset$.
Hence the lower GN-extension is the \emph{least-committal} W-coherent or C-convex
extension of $\lpr$,
a property which identifies the natural or convex natural extension.
\end{proof}

The result of Proposition \ref{pro:GN-extensions} is important,
as it displays a simple way to find the natural extension (and the convex natural extension),
without performing any computation.
For any additional event $C|D$,
we only have to find $(C|D)_*$ or $(C|D)^*$.
This procedure clearly depends on the availability of an initial evaluation on the set of conditional events $\strset$.
This is a special, although not uncommon, case of probability assessment:
it is often customary to evaluate all the events we can obtain from a given partition or universe $\prt$,
i.e. the events of $\strset(\prt)$ in the conditional case.
\begin{example}
Before the final phase of the $20XY$ Football World Cup,
a Swedish bookie elicits a W-coherent upper probability $\upr$ on $\strset(\prt)$ as a basis to fix her odds.
Here $\prt=\{B, S, T\}$,
with $B=\text{`Brazil wins the Cup'}$,
$S=\text{`Sweden wins the Cup'}$,
$T=\text{`A third team wins the Cup'}$.

Later on during the final phase,
the bookie gets to know that event
$F=$`Brazil is qualified for the final game' (and nothing else) is true.
Conditional on $F$,
which is now the largest upper probability for $S$ the bookie can afford,
being consistent with her previous assessment on $\strset(\prt)$?

Since $S|F\notin\strset(\prt)$ ($F$ is not logically dependent on $\prt$),
the problem is that of finding the natural extension of $S|F$,
i.e. $\upr((S|F)^*)$ by Proposition~\ref{pro:GN-extensions}.
Since $(S\wedge F)^*=S$ and $(\nega{S}\wedge F)_*=((B\vee T)\wedge F)_*=B$,
it ensues that $\upr((S|F)^*)=\upr(S|S\vee B)$.

As for the lowest consistent upper probability for $S|F$,
it is $\upr(S|F)=0$,
because $(S|F)_*=\varnothing|B\vee T$ (Proposition \ref{pro:ext_prob}). 
\end{example}

In general, with an assessment $\mu$ on a generic set $\setran$ of
conditional events, we should first extend $\mu$ to some
$\strset(\prt)\supset\setran$ before applying Proposition
\ref{pro:ext_prob_many}.
The GN relation would therefore be of little help,
\emph{operationally}, while remaining
\emph{theoretically} meaningful.
In fact, it still contributes to explain how logical
constraints may determine our inferences.

GN-extensions play a fundamental role also in detecting the other relevant type of extensions recalled in
Section \ref{subsec:imprecise_previsions},
that is \emph{upper extensions}:
\begin{proposition}
\label{pro:upper-extension}
Given $\lpr$ W-coherent or C-convex on $\strset$,
its extension $\lpr(C|D)=\lpr((C|D)^*)$ on $\seteve=\{C|D\}$ is \emph{the} upper extension of $\lpr$.
\end{proposition}
\begin{proof}
By Proposition \ref{pro:ext_prob}
the extension $\lpr(C|D)=\lpr((C|D)^*)$  is W-coherent or C-convex if the starting $\lpr$ is,
and by Proposition \ref{pro:preserve_inequality} any W-coherent or C-convex extension $\underline{Q}$ must satisfy $\underline{Q}(C|D)\leq\underline{Q}((C|D)^*)=\lpr((C|D)^*)$.
\end{proof}

Outside the special case of an extension to a single additional event,
fixed by Proposition \ref{pro:upper-extension},
it is not so simple to determine the upper extension.
Moreover, it is generally no longer unique.

It is also interesting to discuss briefly the effect of Propositions \ref{pro:ext_prob} and  \ref{pro:ext_prob_many} in the special case that the uncertainty measure given on $\strset$ is a dF-coherent probability $P(\cdot|\cdot)$.
Since $P$ is then both a lower and an upper W-coherent probability, we obtain easily
\begin{proposition}
\label{pro:dF-coherent_extensions}
Let $P$ be a dF-coherent probability on $\strset$.
\begin{itemize}
\item[a)] Its dF-coherent extensions to an additional event $C|D$ are those and only those $P(C|D)$ in the closed interval $[P((C|D)_*), P((C|D)^*)]$.
\item[b)] Its extension $\lpr$ ($\upr$) on an arbitrary set of conditional events $\seteve$,
given by the lower  GN-extension $\lpr(C|D)=P((C|D)_*)$
(by the upper GN-extension $\upr(C|D)=P((C|D)^*)$) , $\forall C|D\in\seteve$,
is a W-coherent lower (upper) probability.
\end{itemize}
\end{proposition}
Part a), stated in a \emph{finite} setting in \cite{Col96}, is a
conditional framework version of de Finetti's \emph{Fundamental Theorem of
Probability} (described in \cite{deF74}, but stated already in the
thirties, see \cite{deF37}). Also our previous results for
imprecise probabilities may be viewed as generalisations in an
imprecise setting of the course of reasoning of the Fundamental
Theorem.

Part b) shows that we can extend $P$ using either the lower or the upper GN-extension.
However, the result is generally not a dF-coherent extension,
but a W-coherent imprecise probability.
This is an interesting example of how imprecise assessments may arise from precise ones.

\section{The GN Relation with Imprecise Previsions}
\label{sec:GN_relation_imprecise}
In order to justify a generalisation of the GN relation to conditional gambles,
first introduced in \cite{Pel13},
let us retrace our steps and reconsider the interpretation of the GN relation with conditional events.

\subsection{More on the GN Relation with Conditional Events}
\label{subsec:GN relation conditional events}
Suppose throughout this section that $A|B\lgn C|D$.
As recalled in Section \ref{subsec:GN_relation},
a betting argument has already been
discussed in the literature to justify Definition \ref{def:goongu}.
Precisely, 
\begin{itemize}
\item[a)] whenever $B\wedge D$ is true,
we bet both on $A|B$ and on $C|D$.
If we win the bet on $A|B$ we win also the bet on $C|D$, and,
conversely, losing the bet on $C|D$ implies losing the bet on $A|B$.
\end{itemize}
However, a) is not the only betting implication of $A|B\lgn C|D$.
To see this, note that the partition $\prt_g$ generated
by $A$, $B$, $C$, $D$ allows for at most $7$
non-impossible events that imply $B\vee D$.\footnote{
\label{foo:triviality_events}
We can neglect those events implying $\nega{B}\wedge\nega{D}$
because equation \eqref{eq:goongu} is then trivially satisfied.}
This is easily seen from \eqref{eq:goongu}, using $A\Rightarrow B
\leftrightarrow A\wedge\nega{B}=\varnothing$.
The seven events are:
\begin{equation*}
\label{eq:non-impossible_partition}
\begin{aligned}
&\omega_1 = A\wedge B\wedge C\wedge D,\ \omega_2 =\nega{A}\wedge B\wedge C\wedge D,\ \omega_3=\nega{A}\wedge B\wedge\nega{C}\wedge D;\\
&\omega_4 = A\wedge\nega{B}\wedge C\wedge D,\ \omega_5 =\nega{A}\wedge\nega{B}\wedge C\wedge D;\\
&\omega_6 = \nega{A}\wedge B\wedge C\wedge\nega{D},\ \omega_7 =\nega{A}\wedge B\wedge\nega{C}\wedge\nega{D}.
\end{aligned}
\end{equation*}
%
%

In particular,
one of the impossible elements is $A\wedge B\wedge \nega{C}\wedge D$.
This feature of $\prt_{g}$ ensures the betting implication in a).
In fact, such an implication corresponds to saying that it can never be the case
when $B\wedge D$ holds, that $A$ is true and $C$ false,
and indeed $A\wedge\nega{C}|B\wedge D = A\wedge B\wedge\nega{C}\wedge D|B\wedge D = \varnothing |B \wedge D$.
Note that $B\wedge D=\omega_{1}\vee\omega_{2}\vee\omega_{3}$ in $\prt_g$.
What else can we deduce from realising that $\nega{B}\wedge D=\omega_{4}\vee\omega_{5}$?
Because $C$ is true at both $\omega_4$ and $\omega_5$,
we get $C|\nega{B}\wedge D=C|\omega_{4}\vee\omega_{5}=\Omega|\omega_{4}\vee\omega_{5}$.
Hence a second betting effect:
\begin{itemize}
\item[b)] whenever $\nega{B}\wedge D$ is true,
that is whenever we can bet on $C|D$ while being not allowed to bet on $A|B$,
our winning the bet on $C|D$ is sure.
In fact, this is the same as betting on $C|\nega{B}\wedge D=\Omega|\nega{B}\wedge D$.
\end{itemize}
Similarly, $B\wedge\nega{D}=\omega_6\vee\omega_7$
implies $A|B\wedge\nega{D}=\varnothing|B\wedge\nega{D}$.
Hence, the third betting consequence:
\begin{itemize}
\item[c)] whenever $B\wedge\nega{D}$ is true,
i.e. whenever we can bet on $A|B$ but not on $C|\nega{D}$, the bet
on $A|B$ is surely lost.
In fact, we are actually betting on
$A|B\wedge\nega{D}=\varnothing|B\wedge\nega{D}$.
\end{itemize}
Undoubtedly, the additional betting implications from b) and c) are
very strong.
They help in understanding why the GN relation may be a
very partial order, with several couples of GN non-comparable
conditional events (cf. Remark \ref{rem:incomparable}).
They also highlight other aspects of the GN
relation, which we cannot neglect when extending it to conditional
gambles.

\subsection{The GN relation for conditional gambles}
\label{subsec:GN_conditional_gambles}
With conditional gambles defined on a partition $\prt$,
we introduce the GN relation as follows:
\begin{definition}
\label{def:GN_random_numbers}
$X|B\lgn Y|D$ iff, $\forall\omega\in\prt$,
\begin{equation}
\label{eq:GN_random_numbers}
I_{B}X(\omega)+I_{\nega{B}\wedge D}(\omega)\sup_B X\leq I_{D} Y(\omega)+I_{B\wedge\nega{D}}(\omega)\inf_D Y.
\end{equation}
\end{definition}
To justify this definition, let us verify that its betting implications
are analogous to a), b), c) of the preceding Section \ref{subsec:GN
relation conditional events}.\footnote{
Once again (cf. Footnote \ref{foo:triviality_events}),
we drop the case $\omega\Rightarrow\nega{B}\wedge\nega{D}$,
since then \eqref{eq:GN_random_numbers} holds trivially in the form $0\leq 0$.
}
\begin{itemize}
\item[a')] If $\omega\Rightarrow B\wedge D$, \eqref{eq:GN_random_numbers} reduces to $X(\omega)\leq Y(\omega)$.
This means: whenever we bet both on $X|B$ and on $Y|D$, we gain at
least as much with the bet on $Y|D$;
\item[b')] for $\omega\Rightarrow\nega{B}\wedge D$,
it scales down to $\sup_B X=\sup\{X|B\}\leq Y(\omega)$,
hence $\sup\{X|B\}\leq\inf\{Y|\nega{B}\wedge D\}$. By the last inequality
the gain from betting on $Y|D$, i.e. on $Y|\nega{B}\wedge D$ in this
case, is not less than our (potential) gain on $X|B$, had we bet on
it;
\item[c')] if $\omega\Rightarrow B\wedge\nega{D}$, \eqref{eq:GN_random_numbers} reduces to $X(\omega)\leq\inf_D Y=\inf\{Y|D\}$,
hence $\sup\{X|B\wedge\nega{D}\}\leq\inf\{Y|D\}$. The gain from
betting on $X|B$ while we cannot bet on $Y|D$, i.e. on the gain regarding
$X|B\wedge\nega{D}$ in this case, is dominated by the (potential)
gain on $Y|D$, had we bet on it.
\end{itemize}
In a betting perspective, Definition \ref{def:GN_random_numbers} is a generalisation of Definition \ref{def:condPrev}.
In a sense, c') is even less drastic than c): c') imposes only a dominance condition against the gain regarding $X|B$,
while c) asks for a sure loss when betting on $A|B$.
The differentiation between b') and b) is similar.
Clearly both distinctions depend on the dichotomic nature of events.

When $X$, $Y$ are indicators of events,
say $X=I_{A}$, $Y=I_{C}$,
\eqref{eq:GN_random_numbers} becomes
\begin{equation}
\label{eq:GN_random_numbers_events}
I_{A\wedge B}+I_{\nega{B}\wedge D}\max\{I_{A}|B\}\leq I_{C\wedge D}+I_{B\wedge\nega{D}}\min\{I_{C}|D\}
\end{equation}
and it can be shown that \eqref{eq:GN_random_numbers_events} describes the GN relation
from Definition \ref{def:condPrev} in a less immediate
but equivalent form:
\begin{proposition}
\label{pro:equivalence_GN_rn_events}
When $A|B\neq\varnothing|B$ and
$C|D\neq D|D$, $A|B\lgn C|D$ iff \eqref{eq:GN_random_numbers_events}
holds.
\end{proposition}
\begin{proof}
Let us suppose $A|B\lgn C|D$.
Hence $\nega{C}\wedge D\Rightarrow\nega{A}\wedge B$,
which implies $\nega{B}\wedge\nega{C}\wedge D=\varnothing$.
We get therefore $I_{\nega{B}\wedge D}=I_{\nega{B}\wedge C\wedge D}+I_{\nega{B}\wedge\nega{C}\wedge D}=I_{\nega{B}\wedge C\wedge D}$.
Further, from $A\wedge B\Rightarrow C\wedge D$ and
since $A\wedge B\Rightarrow B$,
we get also $A\wedge B\Rightarrow B\wedge C\wedge D$, i.e. $I_{A\wedge B}\leq I_{B\wedge C\wedge D}$.
It follows
$I_{A\wedge B}+I_{\nega{B}\wedge D}\max\{I_{A}|B\}\leq I_{B\wedge C\wedge D}+I_{\nega{B}\wedge C\wedge D}\max\{I_{A}|B\}\leq I_{B\wedge C\wedge D}+I_{\nega{B}\wedge C\wedge D}=I_{C\wedge D}\leq I_{C\wedge D}+I_{B\wedge\nega{D}}\min\{I_{C}|D\}$.

Conversely, suppose \eqref{eq:GN_random_numbers_events} holds.
Since  $A|B\neq\varnothing|B$ and $C|D\neq D|D$, $\max\{I_{A}|B\}=1$ and $\min\{I_{C}|D\}=0$,
i.e. \eqref{eq:GN_random_numbers_events} becomes $I_{A\wedge B}+I_{\nega{B}\wedge D}\leq I_{C\wedge D}$.
As an immediate consequence, we have that $I_{A\wedge B}\leq I_{C\wedge D}$ or, equivalently, $A\wedge B\Rightarrow C\wedge D$.
To obtain the second implication in \eqref{eq:goongu},
multiply both terms in $I_{A\wedge B}+I_{\nega{B}\wedge D}\leq I_{C\wedge D}$ in turn by $I_{B\wedge\nega{C}\wedge D}$ first and by $I_{\nega{B}\wedge D}$ then.
We get, respectively, $I_{A\wedge B\wedge\nega{C}\wedge D}\leq 0$ and $I_{\nega{B}\wedge D}\leq I_{\nega{B}\wedge C\wedge D}$.
The first inequality implies $A\wedge B\wedge \nega{C}\wedge D=\varnothing$, i.e. $B\wedge \nega{C}\wedge D\Rightarrow\nega{A}$,
the second $\nega{B}\wedge\nega{C}\wedge D=\varnothing$.
It follows $\nega{C}\wedge D=(\nega{B}\wedge\nega{C}\wedge D)\vee(B\wedge\nega{C}\wedge D)=B\wedge\nega{C}\wedge D\Rightarrow\nega{A}\wedge B$.
\end{proof}
\textbf{Convention}
In the sequel,  we may and will always assume the
\emph{non-triviality} conditions $X|B\neq\varnothing|B$ and $Y|D\neq
D|D$, without restrictions: in the cases we are ruling out, any
uncertainty evaluation is trivial by equation~\eqref{eq:consistency_conditions}.

\subsection{Ordering induced by the GN Relation}
\label{subsec:Ordering from GN relation}

As a partial ordering \emph{among conditional gambles},
the generalised GN relation induces an agreeing ordering among their
uncertainty measures whenever they are C-convex or W-coherent
imprecise previsions, or dF-coherent previsions. Since W- and
dF-coherence are special cases of C-convexity, it is enough to
establish the result for C-convex previsions.
\begin{proposition}
\label{pro:GN_measures}
Let $\mu$ be either a lower ($\lpr$) or an upper ($\upr$)
C-convex prevision defined on $\setran\supseteq\{X|B, Y|D \}$.
Then,
\begin{equation}
\label{eq:GN_measures}
X|B\lgn Y|D \rightarrow \mu(X|B)\leq\mu(Y|D).
\end{equation}
\end{proposition}
\begin{proof}
Let $\lpr$ be a C-convex lower prevision on
$\setran$, and consider the special case $n=1$, $s_0=s_1=1$,
$X_0|B_0=Y|D$, $X_1|B_1=X|B$ in Definition \ref{def:C-convexity}.
Correspondingly, we obtain $\lgain=B(X-\lpr(X|B))-D(Y-\lpr(Y|D))=I_{B}X-I_{D}Y+I_{D}\lpr(Y|D)-I_{B}\lpr(X|B)$,
using the notation for the indicators of events in the right-hand term of the second equality.
Assume $X|B\lgn Y|D$,
which ensures by \eqref{eq:GN_random_numbers}
\begin{equation*}
I_{B}X-I_{D}Y\leq I_{B\wedge\nega{D}}\inf_{D}Y-I_{\nega{B}\wedge D} \sup_{B}X.
\end{equation*}

Using the previous inequality,
$\lgain|B\vee D\leq I_{B\wedge\nega{D}}\inf_{D}Y-I_{\nega{B}\wedge D} \sup_{B}X +I_{D}\lpr(Y|D)-I_{B}\lpr(X|B)|B\vee
D\eqdef Z|B\vee D$. Since $\sup\lgain|B\vee D\geq 0$ is necessary
for C-convexity of $\lpr$, it is also necessary that $\max Z|B\vee D
\geq 0$.
But
\[Z|B\vee D=
\begin{cases}
\lpr(Y|D)-\lpr(X|B) &\text{at } B\wedge D \\
-\sup_{B}X+\lpr(Y|D) &\text{at } \nega{B}\wedge D \\
\inf_{D}Y-\lpr(X|B) &\text{at } B\wedge\nega{D}
\end{cases}
\]
Hence, and recalling also that the condition $\lpr(X|B)\in [\inf_{B}X, \sup_{B}X]$ is necessary for C-convexity,
at least one of the following three sets of inequalities must hold:

$\lpr(Y|D)\geq\lpr(X|B)$, or $\lpr(Y|D)\geq\sup_{B}X \geq
\lpr(X|B)$, or $\lpr(Y|D)\geq\inf_{D}Y \geq \lpr(X|B)$.

In all cases then $\lpr(Y|D)\geq\lpr(X|B)$ is necessary for
C-convexity. An analogue proof applies to the case of upper C-convex
previsions.
\end{proof}

\textit{Discussion.}
Proposition \ref{pro:GN_measures} is the most general result in this paper
as for the ordering induced by the GN relation on an uncertainty measure.
Its proof does not extend to (non-centered) convex previsions,
because such previsions may not satisfy the internality condition $\lpr(X|B)\in [\inf_{B}X,\sup_{B}X]$.

Yet, Proposition \ref{pro:GN_measures} still holds under the weaker requirement that
$\lpr$ ($\upr)$ is a \emph{(centered) $1$-convex} lower (upper) prevision.
In the case of lower previsions,
$1$-convex means that $\lpr$ satisfies a modified version of Definition~\ref{def:C-convexity},
where `for all $n\in\nats$' is replaced by `for $n=1$'.
Proposition \ref{pro:GN_measures} holds for $1$-convex lower previsions, because $n=1$ in its proof.
This kind of previsions has not been investigated yet,
to the best of our knowledge.
It extends to conditional gambles the (unconditional) $1$-convex lower previsions introduced and studied in \cite{Bar09},
Section~4.
On their turn, these encompass the notion of \emph{capacity} (normalised, $1$-monotone measure) in the case of events,
and that of \emph{niveloid} \cite{Dol95} for gambles.
Hence, $1$-convex conditional previsions might correspond to some concept of conditional capacity or niveloid.
By Proposition  \ref{pro:GN_measures}, they seem to be the weakest kind of uncertainty measure agreeing with the GN-relation.

Still about (centered) $1$-convex previsions,
note that Proposition~\ref{pro:ext_prob} holds in the case they concern conditional events,
i.e. when they are $1$-convex lower probabilities.
Minor modifications are required in its proof as well as observing that Lemma \ref{lem:extension} can still be used.
Hence, the GN relation ensures a kind of extension theorem for these imprecise probabilities.

Finally,
note that the GN relation induces,
in general,
no agreeing order on measures for conditional gambles that extend the consistency concept of avoiding sure loss (ASL) \cite{Wal91}.
It is easy to realise this considering the unconditional case.
Already in this special instance,
equation \eqref{eq:EimpliesF} does not necessarily hold
if $\mu$ is a lower/upper prevision that avoids sure loss.
In our opinion, this is a counterintuitive,
if not even weak aspect of the ASL concept.

\subsection{Inequalities with the GN Relation}
\label{sec:inequalities_GN}
Stating to what extent the GN relation in Definition \ref{def:GN_random_numbers}
is relevant in extension problems generalising those discussed in Section~\ref{subsec:natural_extension}
is an open question at present.
However, the GN relation may be employed in specific instances for getting
inequalities on uncertainty evaluations.
We introduce the topic in this section.

Unlike conditional gambles, (conditional) events are always non-negative.
It should be expected then that the inequalities regarding events need some sign restriction to be extended to gambles,
or may even be reversed for non-positive gambles.

We see this when trying to generalise equation \eqref{eq:mu_oinequality}
by replacing $A$ with a gamble $X$,
while still $B_1\Rightarrow B_0$.
From the computations displayed in \cite[Example 3]{Pel13} we know that:
\begin{itemize}
\item $X|B_1$ and $B_1 X|B_0$ are GN-comparable iff $\inf(X|B_1)\cdot\sup(X|B_1)\geq 0$,
i.e. iff $X|B_1$ cannot take up values of opposite signs.
\item More specifically,
\begin{equation}
\label{eq:rel_inf}
B_1 X|B_0\lgn X|B_1 \text{ iff } \inf(X|B_1)\geq 0
\end{equation}
\begin{equation}
\label{eq:rel_sup}
X|B_1\lgn B_1 X|B_0  \text{ iff } \sup(X|B_1)\leq 0
\end{equation}
\end{itemize}
Clearly, \eqref{eq:rel_inf} generalises the relation $A\wedge B_{1}|B_{0}\lgn A|B_{1}$, which is a premise to \eqref{eq:mu_oinequality}.
The opposite relation in \eqref{eq:rel_sup} holds for non-positive $X|B_1$.

We may obtain two inequalities for C-convex previsions
from  \eqref{eq:rel_inf} and \eqref{eq:rel_sup},
by means of Proposition \ref{pro:GN_measures}.
For instance,
when $\inf(X|B_1)\geq 0$ we come to
\begin{equation}
\label{eq:cons_inf_sup_1}
\lpr(B_1 X|B_0)\leq\lpr(X|B_1),
\end{equation}
which extends \eqref{eq:mu_oinequality}.
If further $\lpr$ is W-coherent and $\lpr(X|B_1)\cdot\lpr(B_1|B_0)>0$,
equation \eqref{eq:cons_inf_sup_1} can be matched with the upper bound \eqref{eq:product_rule_1},
getting
\begin{equation}
\label{eq:cons_inf_sup_2}
\lpr(B_1 X|B_0)\leq\lpr(X|B_1)\leq\lpr(B_1 X|B_0)/\lpr(B_1|B_0).
\end{equation}
We give a further example of inequality
derived using the GN relation and some elementary properties of W-coherence.
Let us make for this the following assumptions:
\begin{itemize}
\item[a)] A gamble $X$  and a partition $\prt$ are given.
Further, a W-coherent lower prevision $\lpr$ is assigned on $\strset(\prt)\cup\{X|D:D\in\genfield, D \neq\varnothing\}$.
\item[b)] $B$ is a non-impossible event
such that $B\notin\strset(\prt)$.
\end{itemize}
We would like to bound the uncertainty evaluation of $X|B$.
For this, define the instrumental gamble $Y=X-\inf_{B}X$.
Thus, $Y|B\geq 0$, $Y|B_{*}\geq 0$ and, by W-coherence, $\lpr(Y|B_{*})\geq 0$.

Suppose now $\lpr(X|B_{*})>\inf_{B}X$, i.e. $\lpr(Y|B_{*})>0$.\footnote{
We use here the W-coherence property $\lpr(X+h|B)=\lpr(X|B)+h,\ \forall h\in\reals$.}
Since $B_{*}Y|B\leq Y|B$ and using \eqref{eq:product_rule_1},
we may write
\begin{equation}
\label{eq:chain_ineq_1}
\lpr(Y|B)\geq\lpr(B_* Y|B)\geq\lpr(B_*|B)\cdot\lpr(Y|B_*).
\end{equation}
In the rightmost term of \eqref{eq:chain_ineq_1} we know
$\lpr(Y|B_*)=\lpr(X|B_*)-\inf_{B}X$ (by a)),
but not $\lpr(B_*|B)$.
However,
$\lpr(B_*|B)\geq\lpr((B_*|B)_*)$  by Proposition \ref{pro:preserve_inequality}.
Recalling eq. \eqref{eq:max_cond_event},
and since $\bigvee\limits_{e\in\prt:\ e\Rightarrow B_*\wedge B} e=\bigvee\limits_{e\in\prt:\ e\Rightarrow B_{*}} e=B_*$,
we have $(B_*|B)_*=B_*|B^*$.
Using these facts in \eqref{eq:chain_ineq_1},
we easily obtain the inequality
\begin{equation}
\label{eq:chain_ineq_2}
\lpr(X|B)\geq\lpr(B_*|B^*)\cdot\lpr(X|B_*)+\upr(\nega{B_{*}}|B^{*})\cdot\inf_{B} X
\end{equation}
where the uncertainty evaluations in the right-side term are known.

When $\lpr(X|B_{*})=\inf_{B}X$, i.e. $\lpr(Y|B_{*})=0$,
\eqref{eq:chain_ineq_2} holds trivially.
In fact, it reduces to $\lpr(X|B)\geq\inf_{B}X$,
a condition implied by W-coherence.

Similar inequalities may be obtained for upper previsions and/or making specific assumptions.

For instance, if it is further supposed that $X|B$ takes \emph{finitely} many distinct values $x_1,x_2,\ldots,x_n\geq 0$,
the following bound holds (see \cite{Pel13}), where $\omega_i=(X=x_i)$:
\begin{equation}
\label{eq:chain_ineq_3}
\lpr(X|B)\geq\sum_{i=1}^{n}x_i\lpr((\omega_i|B)_*).
\end{equation}

\section{Conclusions}
\label{sec:conclusions}
The analysis of the GN relation within Imprecise Probability Theory carried out in this paper
shows that it preserves the basic monotonicity property (equation \eqref{eq:EimpliesF})
of implication towards several kinds of imprecise conditional previsions,
including the (weak) consistency concept of $1$-convexity.
This latter notion has still to be focalised for most of its aspects.
The role of the GN relation in extension problems is fixed by the results in Section \ref{subsec:natural_extension},
as for conditional events.
It largely remains a topic for future work in the case of conditional gambles.

\section*{Acknowledgements}
We wish to thank the referees for their constructive suggestions.

*NOTICE: This is the authors' version of a work that was accepted for publication in the International Journal of Approximate Reasoning. Changes resulting from the publishing process, such as peer review, editing, corrections, structural formatting, and other quality control mechanisms may not be reflected in this document. Changes may have been made to this work since it was submitted for publication. A definitive version was subsequently published in the International Journal of Approximate Reasoning, vol. 55, issue 8, November 2014, pages 1694-–1707, doi:10.1016/j.ijar.2014.06.002

$\copyright$ Copyright Elsevier

http://www.sciencedirect.com/science/article/pii/S0888613X14001017


\bibliographystyle{amsplain}
\bibliography{refs}{}

\end{document}

%% file: macros.tex
\newcommand{\reals}{\mathbb{R}}
\newcommand{\nats}{\mathbb{N}}

\newcommand{\pr}{P}
\newcommand{\lpr}{{\underline{\pr}}}
\newcommand{\upr}{{\overline{\pr}}}
\newcommand{\nex}{E}
\newcommand{\lnex}{{\underline{\nex}}}

\newcommand{\lgain}{{\underline{G}}}

\newcommand{\field}{\mathcal{A}}
\newcommand{\seteve}{\mathcal{E}}
\newcommand{\genfield}{\field(\prt)}
\newcommand{\strset}{\field_{C}}
\newcommand{\setran}{\mathcal{S}}

\newcommand{\nega}[1]{\neg #1}

\newcommand{\dsn}{\!\!}
\newcommand{\prt}{I\dsn P}

\newcommand{\eqdef}{\stackrel{\mathrm{def}}{=}}

\newcommand{\lgn}{\leq_{\textsc{GN}}}

{\left\lbrace\begin{array}{@{}l@{}}}%
{\end{array}\right.}